%% file: stability.tex
\documentclass{amsart}

\newif\ifcref\creftrue
\input{decls}

\tikzset{lab/.style={auto,font=\scriptsize}} 
\usetikzlibrary{arrows}

\usepackage{xcolor}
\usepackage{fixme}
\fxsetup{
    status=draft,
    author=,
    layout=margin,
    theme=color
}

\definecolor{fxnote}{rgb}{1.0000,0.0000,0.0000}
\colorlet{fxnotebg}{yellow}

\newcommand{\tw}{\ensuremath{\operatorname{tw}}}
\newcommand{\D}{\sD}
\newcommand{\E}{\sE}

\newcommand{\V}{\sV}

\autodefs{\cDER\cCat\cGrp\cCAT\cMONCAT\cTriaCAT\cAddCAT\cPreCAT\cPtCAT\bSet\cProf\bCat
\cPro\cMONDER\cBICAT\ncomp\nid\ncoker\niso\ndia\nIm\sProf\ncyl\fC\fZ\fT\nhocolim\nholim\cPsNat\ncoll\nCh\bsSet\cAlg\cIm\cI\cP}

\def\cLDER{\ensuremath{\mathcal{LD}\mathit{ER}}\xspace}

\let\oldboxtimes\boxtimes
\def\boxtimes{\mathrel{\oldboxtimes}}

\newcommand{\fib}{\mathsf{fib}}
\newcommand{\cof}{\mathsf{cof}}

\def\ccsub{_{\mathrm{cc}}}
\newcommand{\pdh}{\mathsf{HOM}}
\newcommand{\ldh}{\mathsf{HOM}\ccsub}

\newcommand{\lend}{\mathsf{END}\ccsub}

\def\DTl#1#2#3#4#5#6#7{%
  \xymatrix@C=3pc{{#1} \ar[r]^-{#2} &
    {#3} \ar[r]^-{#4} &
    {#5} \ar[r]^-{#6} &
    {#7}
  }}

\newsavebox{\tvabox}
\savebox\tvabox{\hspace{1mm}\begin{tikzpicture}[>=latex',baseline={(0,-.18)}]
  \draw[->] (0,.1) -- +(1,0);
  \node at (.5,0) {$\scriptscriptstyle\bot$};
  \draw[->] (1,-.1) -- +(-1,0);
  \draw[->] (1,-.2) -- +(-1,0);
\end{tikzpicture}\hspace{1mm}}

\newcommand{\tcof}{\mathsf{tcof}}
\newcommand{\tfib}{\mathsf{tfib}}
\newcommand{\cok}{\mathrm{cok}}

\newtheorem*{thm*}{\textbf{Theorem}}

\newcounter{itemgroup}
\makeatletter
\newcommand{\newitemgroup}{\setcounter{itemgroup}{0}\@itemgroup}
\newcommand{\itemgroup}{\addtocounter{enumi}{-1}\@itemgroup}
\newcommand{\@itemgroup}{%
  \stepcounter{itemgroup}%
  \renewcommand{\theenumi}{(\roman{enumi}.\alph{itemgroup})}%
  \renewcommand{\labelenumi}{\theenumi}%
  \item %
}
\newcommand{\stopitemgroup}{%
  \renewcommand{\theenumi}{(\roman{enumi})}%
  \renewcommand{\labelenumi}{\theenumi}%
}
\makeatother

\newcommand{\stabl}{\mathsf{Stab}_L}
\newcommand{\absl}{\mathsf{Abs}_L}
\newcommand{\stabr}{\mathsf{Stab}_R}
\newcommand{\absr}{\mathsf{Abs}_R}

\title{Generalized stability for abstract homotopy theories}

\author{Moritz Groth and Michael Shulman}
\thanks{This material is based on research sponsored by The United States Air Force Research Laboratory under agreement number FA9550-15-1-0053.  The U.S. Government is authorized to reproduce and distribute reprints for Governmental purposes notwithstanding any copyright notation thereon.  The views and conclusions contained herein are those of the authors and should not be interpreted as necessarily representing the official policies or endorsements, either expressed or implied, of the United States Air Force Research Laboratory, the U.S. Government, or Carnegie Mellon University.}
\date{\today}

\begin{document}

\begin{abstract}
  We show that a derivator is stable if and only if homotopy finite limits and homotopy finite colimits commute, if and only if homotopy finite limit functors have right adjoints, and if and only if homotopy finite colimit functors have left adjoints.
  These characterizations generalize to an abstract notion of ``stability relative to a class of functors'', which includes in particular pointedness, semiadditivity, and ordinary stability.
  To prove them, we develop the theory of derivators enriched over monoidal left derivators and weighted homotopy limits and colimits therein.
\end{abstract}

\maketitle

\tableofcontents

\section{Introduction}
\label{sec:intro}

In classical algebraic topology we have the following pair of adjunctions relating topological spaces $\mathrm{Top}$ to pointed spaces $\mathrm{Top}_\ast$ and spectra $\mathrm{Sp}$:
\[
(\Sigma^\infty_+,\Omega^\infty_-)\colon\mathrm{Top}\rightleftarrows\mathrm{Top}_\ast\rightleftarrows\mathrm{Sp}
\]
Abstractly, each of these two steps universally improves certain \emph{exactness properties} of a homotopy theory. In the first step we pass in a universal way from a general homotopy theory to a \emph{pointed} homotopy theory, i.e., a homotopy theory admitting a zero object. The second step realizes the universal passage from a pointed homotopy theory to a \emph{stable} homotopy theory, i.e., to a pointed homotopy theory in which homotopy pushouts and homotopy pullbacks coincide. With this in mind, our first goal in this paper is to collect additional answers to the following question.

\vspace{2mm}
\textbf{Question:} Which exactness properties of the homotopy theory of spectra already \emph{characterize} the passage from (pointed) topological spaces to spectra? To put it differently, starting with the homotopy theory of (pointed) topological spaces, for which exactness properties is it true that if one imposes these properties in a universal way then the outcome is the homotopy theory of spectra?
\vspace{2mm}

To make this question precise, we need a definition of an ``abstract homotopy theory''; here we choose to work with derivators. (However,  similar arguments should also apply to $\infty$-categories.) For the introduction it suffices to know that derivators provide some framework for the calculus of homotopy limits, colimits, and Kan extensions as it is available in typical situations arising in homological algebra and abstract homotopy theory (see e.g.\ \cite{groth:intro-to-der-1} for more details).

A derivator is by definition \emph{stable} if it admits a zero object (i.e.\ it is pointed) and if the classes of pullback squares and pushout squares coincide. Typical examples are given by derivators of unbounded chain complexes in Grothendieck abelian categories (like derivators associated to fields, rings, or schemes), and homotopy derivators of stable model categories or stable $\infty$-categories (see \cite[\S5]{gst:basic} for many explicit examples). The ``universal'' example is the derivator of spectra, which is obtained by stabilizing the derivator of spaces \cite{heller:stable}.

It is known that stability can be reformulated by asking that the derivator is pointed and that the suspension-loop adjunction or the cofiber-fiber adjunction is an equivalence \cite{gps:mayer}. Alternatively, by \cite{gst:basic} a pointed derivator is stable exactly when the classes of strongly cartesian $n$-cubes (in the sense of Goodwillie \cite{goodwillie:II}) and strongly cocartesian $n$-cubes agree for all $n\geq 2$. 

Our first new characterization in this paper is that stable derivators are precisely those derivators in which homotopy finite limits and homotopy finite colimits commute. (A category is ``homotopy finite'' if it is equivalent to a category which is finite, skeletal, and has no non-trivial endomorphisms, i.e., to a category whose nerve is a finite simplicial set.) Since Kan extensions in derivators are pointwise, these characterizations admit various improvements in terms of the commutativity of Kan extensions. This gives \autoref{thm:stable-lim-III}:

\begin{thm*}
The following are equivalent for a derivator \D.
\begin{enumerate}
\item The derivator \D is stable.
\item The derivator \D is pointed and the cone morphism $C\colon\D^{[1]}\to\D$ preserves fibers. (Here, $\D^{[1]}$ denotes the derivator of morphisms in \D.)
\item Homotopy finite colimits and homotopy finite limits commute in \D.
\item Left homotopy finite left Kan extensions commute with arbitrary right Kan extensions in \D.\label{item:il}
\item Arbitrary left Kan extensions commute with right homotopy finite right Kan extensions in \D. \label{item:ir}
\end{enumerate}
\end{thm*}

Since the derivator of spectra is the stabilization of the derivator of spaces, these abstract characterizations of stability  specialize to answers to the above question.

\vspace{2mm}
\textbf{Answer \#1:} The homotopy theory of spectra is obtained from that of spaces if one forces homotopy finite limits and homotopy finite colimits to commute in a universal way. 
\vspace{2mm}

Characterizations~\ref{item:il} and~\ref{item:ir} in the above theorem suggest a natural generalization: if $\Phi$ is any class of functors between small categories, we define a derivator \D to be \emph{left $\Phi$-stable} if left Kan extensions along functors in $\Phi$ commute with arbitrary right Kan extensions in \D, and dually \emph{right $\Phi$-stable}.
For instance, stable derivators are precisely the left $\mathsf{FIN}$-stable derivators and also the right $\mathsf{FIN}$-stable derivators, where $\mathsf{FIN}$ is the class of homotopy finite categories (more precisely, the class of the corresponding functors to the terminal category).
But other interesting stability properties also arise in this way; for instance, pointed derivators are precisely the left or right $\{\emptyset\}$-stable ones (i.e.\ initial objects commute with right Kan extensions, or terminal objects commute with left Kan extensions).
And semi-additive derivators are precisely the left or right $\mathsf{FINDISC}$-stable ones, where $\mathsf{FINDISC}$ is the class of finite discrete categories.
In general, this notion of ``relative stability'' yields a Galois connection between collections of derivators and classes of functors.

To understand relative stability better, we introduce \emph{enriched} derivators and weighted colimits.
These build on the theory of monoidal derivators developed in~\cite{gps:additivity,ps:linearity}, extending the classical theory of enriched categories to the context of derivators.
Just as every ordinary category is enriched over the category of sets, every derivator is enriched\footnote{In a weak sense; see below.} over the derivator of spaces; whereas pointed derivators are automatically enriched over pointed spaces, and stable ones over spectra.
For any \V-enriched derivator we have a notion of limit or colimit weighted by ``profunctors'' in \V, which includes the ordinary homotopy Kan extensions that exist in any derivator.

With the technology of enriched derivators, we can prove the following general characterization of relative stability (\cref{thm:stab-op}):

\begin{thm*}
  The following are equivalent for a derivator \D and a class $\Phi$ of functors.
  \begin{enumerate}[label=(\alph*)]
  \item \D is left $\Phi$-stable, i.e.\ left Kan extensions along functors in $\Phi$ commute with arbitrary right Kan extensions in \D.
  \item \D is right $\Phi\op$-stable, i.e.\ right Kan extensions along functors in $\Phi\op$ commute with arbitrary left Kan extensions in \D.
  \item Left Kan extension functors $u_! : \D^A \to \D^B$ for functors $u\in \Phi$ are right adjoint morphisms of derivators.
  \item Right Kan extension functors $(u\op)_\ast : \D^{A\op} \to \D^{B\op}$ for functors $u\in \Phi$ are weighted \emph{colimit} functors relative to some \V over which \D is enriched.\label{item:ie}
  \end{enumerate}
\end{thm*}

This gives some additional conceptual explanations for why certain limits and colimits commute: if a colimit functor is a right adjoint, then of course it commutes with all limits; whereas if a limit functor can be identified with a (weighted) \emph{colimit} functor, then of course it commutes with all other colimits.
It also explains the left-right duality in the first theorem as due to the fact that the class $\mathsf{FIN}$ of finite categories is closed under taking opposites.
Thus we can say:

\vspace{2mm}
\textbf{Answer \#2:} The homotopy theory of spectra is obtained from that of spaces by universally forcing homotopy finite limits to be weighted \emph{colimits}, and dually.
\vspace{2mm}

There is one fly in the ointment: the ``enrichment'' in~\ref{item:ie} is rather weak: it has only tensors and not cotensors or hom-objects (so it is more properly called simply a ``\V-module'' rather than a ``\V-enriched derivator''), and moreover \V is not itself a derivator, only a ``left derivator'' (having left homotopy Kan extensions but not right ones).
This can be remedied by working with locally presentable $\infty$-categories rather than derivators, which we plan to do in~\cite{gs:enriched}.
However, this depends on rather more technical machinery, so it is interesting how much can be done purely in the realm of derivators.

In~\cite{gs:enriched} we will also show more, namely that given $\Phi$ there is a \emph{universal} choice of $\V$ in~\ref{item:ie}, with pointed spaces and spectra being particular examples.
The construction again depends on the good behavior of local presentability, so it seems unlikely to hold in general for derivators.
However, as noted above, in particular cases such a universal derivator does exist, such as pointed spaces and spectra for the cases $\Phi=\{\emptyset\}$ and $\Phi=\mathsf{FIN}$ respectively.
For $\Phi=\mathsf{FINDISC}$ we expect that the universal \V consists of $E_\infty$-spaces, though we have not proven this.

This paper belongs to a project aiming for an abstract study of stability, and can be thought of as a sequel to \cite{groth:ptstab,gps:mayer,groth:can-can,ps:linearity} and as a prequel to \cite{gs:enriched}. This abstract study of stability was developed in a different direction in the series of papers on abstract representation theory \cite{gst:basic,gst:tree,gst:Dynkin-A,gst:acyclic} which will be continued in \cite{gst:acyclic-Serre}. The perspective from enriched derivator theory offers additional characterizations of stability, and these together with a more systematic study of the stabilization will appear in \cite{gs:enriched}.
It is worth noting that in~\cite{ps:linearity}, what we here call ``$\Phi$-stable monoidal derivators'' are shown to admit a linearity formula for the traces and Euler characteristics of $\Phi$-colimits, so the abstract study of stability has computational as well as conceptual importance.

The content of the paper is as follows. In \S\ref{sec:char} we characterize pointed and stable derivators by the commutativity of certain (co)limits or Kan extensions. In \cref{sec:galois} we define the Galois correspondence of relative stability.  In \cref{sec:enriched-derivators} we define enriched derivators and weighted colimits, and in \cref{sec:stab-via-wcolim} we use them to give the second class of characterizations of stability.  Finally, in \S\ref{sec:fun} we study further the characterizations in terms of iterated adjoints to constant morphism morphisms.
\vspace{2mm}

\textbf{Prerequisites.} We assume \emph{basic} acquaintance with the language of derivators, which were introduced independently by Grothendieck~\cite{grothendieck:derivators}, Heller~\cite{heller:htpythies}, and Franke~\cite{franke:adams}. Derivators were developed further by various mathematicians including Maltsiniotis \cite{maltsiniotis:seminar,maltsiniotis:k-theory,maltsiniotis:htpy-exact} and Cisinski \cite{cisinski:direct,cisinski:loc-min,cisinski:derived-kan} (see \cite{grothendieck:derivators} for many additional references). Here we stick to the notation and conventions from \cite{gps:mayer}. For a more detailed account of the basics we refer to \cite{groth:intro-to-der-1}.

\section{Stability and commuting (co)limits}
\label{sec:char}

In this section we obtain characterizations of pointed and stable derivators in terms of the commutativity of certain left and right Kan extensions. It turns out that a derivator is stable if and only if homotopy finite colimits and homotopy finite limits commute, and there are variants using suitable Kan extensions.  

We begin by collecting the following characterizations which already appeared in the literature.

\begin{thm}\label{thm:stable-known}
The following are equivalent for a pointed derivator \D.
\begin{enumerate}
\item The adjunction $(\Sigma,\Omega)\colon\D\rightleftarrows\D$ is an equivalence.
\item The derivator \D is $\Sigma$-stable, i.e., a square in \D is a suspension square if and only if it is a loop square.
\item The adjunction $(\cof,\fib)\colon\D^{[1]}\rightleftarrows\D^{[1]}$ is an equivalence.
\item The derivator \D is cofiber-stable, i.e., a square in \D is a cofiber square if and only if it is a fiber square.
\item The derivator \D is stable, i.e., a square in \D is cocartesian if and only if it is cartesian.
\item An $n$-cube in \D, $n\geq 2,$ is strongly cocartesian if and only if it is strongly cartesian.
\end{enumerate}
\end{thm}
\begin{proof}
The equivalence of the first five statements is \cite[Thm.~7.1]{gps:mayer} and the equivalence of the remaining two is \cite[Cor.~8.13]{gst:tree}.
\end{proof}

As a preparation for a minor variant we include the following construction.

\begin{con}
In every pointed derivator \D there are canonical comparison maps
\begin{equation}\label{eq:sigma-f-c}
\Sigma F\to C\colon\D^{[1]}\to\D\qquad\text{\and}\qquad F\to \Omega C\colon\D^{[1]}\to\D.
\end{equation}
In fact, starting with a morphism $(f\colon x\to y)\in\D^{[1]}$ we can pass to the coherent diagram encoding both the corresponding fiber and cofiber square,
\[
\xymatrix{
Ff\ar[r]\ar[d]\pullbackcorner&x\ar[d]^-f\ar[r]&0\ar[d]\\
0\ar[r]&y\ar[r]&Cf.\pushoutcorner
}
\]
More formally, let $i\colon[1]\to\boxbar=[2]\times[1]$ classify the vertical morphism in the middle and let 
\[
i\colon [1]\stackrel{i_1}{\to}A_1\stackrel{i_2}{\to}A_2\stackrel{i_3}{\to}A_3\stackrel{i_4}{\to}\boxbar
\]
be the fully faithful inclusions which in turn add the objects $(2,0),(2,1),(0,1),$ and $(0,0)$. In every pointed derivator we can consider the corresponding Kan extension morphisms
\[
\D^{[1]}\stackrel{(i_1)_\ast}{\to}\D^{A_1}\stackrel{(i_2)_!}{\to}\D^{A_2}\stackrel{(i_3)_!}{\to}\D^{A_3}\stackrel{(i_4)_\ast}{\to}\D^\boxbar.
\]
The first two functors add a cofiber square and homotopy (co)finality arguments (for example based on \cite[Prop.~3.10]{groth:ptstab}) show that the remaining two morphisms add the fiber square. Forming the composite square, we obtain a coherent square looking like
\begin{equation}\label{eq:F-C-square}
\vcenter{
\xymatrix{
Ff\ar[r]\ar[d]&0\ar[d]\\
0\ar[r]&Cf.
}
}
\end{equation}
The canonical comparison maps \eqref{eq:sigma-f-c} result from considering suitable loop and suspension squares.
\end{con}

\begin{prop}\label{prop:stable-known-mod}
The following are equivalent for a pointed derivator \D.
\begin{enumerate}
\item The derivator \D is stable.\label{item:sk1}
\item For every $f\in\D^{[1]}$ the canonical comparison maps $\Sigma F\to C$ and $F\to \Omega C$ as in \eqref{eq:sigma-f-c} are isomorphisms.\label{item:sk2}
\item For every $f\in\D^{[1]}$ the square \eqref{eq:F-C-square} is bicartesian.\label{item:sk3}
\end{enumerate}
\end{prop}
\begin{proof}
If \D is a stable derivator, then the composition property of bicartesian squares \cite[Prop.~3.13]{groth:ptstab} implies that \eqref{eq:F-C-square} is bicartesian, and it follows from \cite[Prop.~2.16]{groth:can-can} that the canonical transformations $\Sigma F\to C$ and $F\to\Omega C$ are invertible. It remains to show that~\ref{item:sk2} implies~\ref{item:sk1}, and we hence assume that $\Sigma F\toiso C$ is invertible. Associated to $x\in\D$ there is by \cite[Prop.~3.6]{groth:ptstab} the morphism $1_!(x)=(0\to x)\in\D^{[1]}$. The natural isomorphism $\Sigma F\toiso C$ evaluated at $1_!(x)$ yields a natural isomorphism $\Sigma\Omega x\toiso x$. Dually, we deduce $\id\toiso\Omega\Sigma$ and \autoref{thm:stable-known} concludes the proof.
\end{proof}

While unrelated left Kan extensions always commute \cite[Cor.~4.3]{groth:can-can}, it is, in general, not true that unrelated left and right Kan extensions commute. More specifically, given functors $u\colon A\to A'$ and $v\colon B\to B'$, recall that \textbf{left Kan extension along $u$ and right Kan extension along $v$ commute} in a derivator \D if the canonical mate
\begin{align*}
(u\times\id)_!(\id\times v)_\ast&\stackrel{\eta}{\to} (u\times\id)_!(\id\times v)_\ast (u\times\id)^\ast(u\times\id)_!\\
&\toiso (u\times\id)_!(u\times\id)^\ast(\id\times v)_\ast (u\times\id)_!\\
& \stackrel{\varepsilon}{\to} (\id\times v)_\ast (u\times\id)_!
\end{align*}
is an isomorphism in \D. This is to say that the morphism $u_!\colon\D^A\to\D^{A'}$ preserves right Kan extensions along $v$ or that the morphism $v_\ast\colon\D^B\to\D^{B'}$ preserves left Kan extensions along $u$ \cite[Lem.~4.8]{groth:can-can}. For the purpose of a simpler terminology, we also say that $u_!$ and $v_\ast$ commute in \D. 

In general, these canonical mates are not invertible as is for example illustrated by the following characterization of pointed derivators.

\begin{prop}\label{prop:ptd-comm}
The following are equivalent for a derivator \D.
\begin{enumerate}
\item The derivator \D is pointed.\label{item:pc1}
\item Empty colimits and empty limits commute in \D.\label{item:pc2}
\item Left Kan extensions along cosieves and right Kan extensions along sieves commute in \D.\label{item:pc3}
\newitemgroup Left Kan extensions along cosieves and arbitrary right Kan extensions commute in \D. \label{item:pc4a}
\itemgroup Arbitrary left Kan extensions and right Kan extensions along sieves commute in \D.\label{item:pc4b}
\stopitemgroup
\end{enumerate}
\end{prop}
\begin{proof}
For the equivalence of the first two statements we consider the empty functor $\emptyset\colon\emptyset\to\bbone$. Correspondingly, for every derivator \D there is the canonical mate
\[
\xymatrix{
\D^{\emptyset\times\emptyset}\ar[r]^-{(\id\times\emptyset)_\ast}\ar[d]_-{(\emptyset\times\id)_!}\drtwocell\omit{}&
\D^{\emptyset\times\bbone}\ar[d]^--{(\emptyset\times\id)_!}\\
\D^{\bbone\times\emptyset}\ar[r]_--{(\id\times\emptyset)_\ast}&\D^{\bbone\times\bbone}
}
\]
detecting if empty colimits and empty limits commute. By construction of initial and final objects in derivators (see \cite[\S1.1]{groth:ptstab}), the source of this canonical mate is given by initial objects in \D while the target is given by final objects. Hence, \D is pointed if and only if empty colimits and empty limits commute in \D.

Obviously, each of the statements~\ref{item:pc4a} or~\ref{item:pc4b} implies statement~\ref{item:pc3}. Moreover, since the empty functor is a sieve and a cosieve, statement~\ref{item:pc3} implies~\ref{item:pc2}. By duality, it remains to show that~\ref{item:pc1} implies~\ref{item:pc4a}. Given a functor $u\colon A\to B$, the morphism $u_\ast\colon\D^A\to\D^B$ is a right adjoint and, as a pointed morphism of pointed derivators, $u_\ast$ preserves left Kan extensions along cosieves \cite[Cor.~8.2]{groth:can-can}.
\end{proof}

We now turn to the stable context. Let us recall that a category $A\in\cCat$ is \textbf{strictly homotopy finite} if it is finite, skeletal, and it has no non-trivial endomorphisms (equivalently the nerve $NA$ is a finite simplicial set). A category is \textbf{homotopy finite} if it is equivalent to a strictly homotopy finite category.

\begin{thm}\label{thm:stable-lim-I}
Homotopy finite colimits and homotopy finite limits commute in stable derivators.
\end{thm}
\begin{proof}
Let \D be a stable derivator and let $A\in\cCat$. Denoting by $\pi_A\colon A\to\bbone$ the unique functor, there are defining adjunctions 
\[
(\colim_A,\pi_A^\ast)\colon\D^A\rightleftarrows\D\qquad\text{and}\qquad (\pi_A^\ast,\mathrm{lim}_A)\colon\D\rightleftarrows\D^A,
\]
and these exhibit $\colim_A,\mathrm{lim}_A\colon\D^A\to\D$ as exact morphisms of stable derivators \cite[Cor.~9.9]{groth:can-can}. Hence, by \cite[Thm.~7.1]{ps:linearity}, $\colim_A$ preserves homotopy finite limits and $\lim_A$ preserves homotopy finite colimits.
\end{proof}

For the converse to this theorem we collect the following lemma. 

\begin{lem}\label{lem:lim-comm}
Let \D be a derivator such that homotopy finite colimits and homotopy finite limits commute in \D.
\begin{enumerate}
\item The derivator \D is pointed.
\item The morphisms $\cof\colon\D^{[1]}\to\D^{[1]}$ and $C\colon\D^{[1]}\to\D$ preserve homotopy finite limits.
\item The morphism $\fib\colon\D^{[1]}\to\D^{[1]}$ and $F\colon\D^{[1]}\to\D$ preserve homotopy finite colimits.
\end{enumerate}
\end{lem}
\begin{proof}
By assumption on \D, empty colimits and empty limits commute and this implies that \D is pointed (\autoref{prop:ptd-comm}). Hence, by duality, it remains to take care of the second statement. Denoting by $i\colon[1]\to\ulcorner$ the sieve classifying the horizontal morphism $(0,0)\to (1,0)$ and by $k'\colon[1]\to\square$ the functor classifying the vertical morphism $(1,0)\to (1,1)$, the cofiber morphism is given by
\begin{equation}\label{eq:cof}
\cof\colon\D^{[1]}\stackrel{i_\ast}{\to}\D^\ulcorner\stackrel{(i_\ulcorner)_!}{\to}\D^\square\stackrel{(k')^\ast}{\to}\D^{[1]}.
\end{equation}
Since the morphisms $i_\ast$ and $(k')^\ast$ are right adjoints, they preserve arbitrary right Kan extensions, hence homotopy finite limits. By assumption on \D, \cite[Prop.~3.9]{groth:can-can}, and \cite[Lem.~4.9]{groth:can-can}, also the morphism $(i_\ulcorner)_!$ preserves homotopy finite limits, and hence so does $\cof$ by \cite[Prop.~5.2]{groth:can-can}. An additional composition with the continuous evaluation morphism $1^\ast\colon\D^{[1]}\to\D$ establishes the corresponding result for $C$. 
\end{proof}

Given a pointed derivator \D, the derivator $\D^\square=\D^{[1]\times[1]}$ admits cone and fiber morphisms in the first and the second coordinate, and these are respectively denoted by
\[
C_1, C_2\colon\D^\square\to\D^{[1]}\qquad\mbox{and}\qquad F_1,F_2\colon\D^\square\to\D^{[1]}.
\]
Since these morphisms are pointed, for $X\in\D^\square$ there is by \cite[Construction~9.7]{groth:can-can} a canonical comparison map
\begin{equation}\label{eq:cof-fib-comm}
C(F_2 X)\to F(C_1X).
\end{equation}

\begin{cor}\label{cor:lim-comm}
Let \D be a derivator in which homotopy finite colimits and homotopy finite limits commute. Then \D is pointed and the canonical transformations \eqref{eq:cof-fib-comm} are isomorphisms for every $X\in\D^\square$.
\end{cor}
\begin{proof}
This is immediate from \autoref{lem:lim-comm}.
\end{proof}

As we show next, this property already implies that the derivator is stable. Together with \autoref{thm:stable-lim-I} we thus obtain the following more conceptual characterization of stability.

\begin{thm}\label{thm:stable-lim-II}
A derivator is stable if and only if homotopy finite colimits and homotopy finite limits commute.
\end{thm}
\begin{proof}
By \autoref{thm:stable-lim-I} it suffices to show that a derivator \D is stable as soon as homotopy finite colimits and homotopy finite limits commute in \D. Such a derivator is pointed and for every $X\in\D^\square$ the canonical morphism
\begin{equation}\label{eq:stable-lim-II}
C(F_2X)\toiso F(C_1X)
\end{equation}
is an isomorphism (\autoref{cor:lim-comm}). For every $x\in\D$ we consider the square
\[
X=X(x)=(i_\lrcorner)_!\pi_\lrcorner^\ast x\in\D^\square.
\]
The morphism $\pi_\lrcorner^\ast\colon\D\to\D^\lrcorner$ forms constant cospans. Since $i_\lrcorner\colon\lrcorner\to \square$ is a cosieve, $(i_\lrcorner)_!$ is left extension by zero \cite[Prop.~3.6]{groth:ptstab} and the diagram $X\in\D^\square$ looks like
\[
\xymatrix{
0\ar[r]\ar[d]&x\ar[d]^-\id\\
x\ar[r]_-\id&x.
}
\]
We calculate $CF_2(X)\cong C(\Omega x\to 0)\cong \Sigma\Omega x$ and $FC_1(X)\cong F(x\to 0)\cong x$, showing that the canonical isomorphism \eqref{eq:stable-lim-II} induces a natural isomorphism $\Sigma\Omega\toiso\id$. Using constant spans instead one also constructs a natural isomorphism $\id\toiso\Omega\Sigma$, showing that $\Sigma,\Omega\colon\D\to\D$ are equivalences. It follows from \autoref{thm:stable-known} that \D is stable.
\end{proof}

It is now straightforward to obtain the following variant of this theorem. We recall from \cite[\S9]{groth:can-can} that \textbf{left homotopy finite left Kan extensions} are left Kan extensions along functors $u\colon A\to B$ such that the slice categories $(u/b),$ $b\in B,$ admit a homotopy final functor $C_b\to(u/b)$ from a homotopy finite category~$C_b$. The point of this notion is that right exact morphisms of derivators preserve left homotopy finite left Kan extensions \cite[Thm.~9.14]{groth:can-can}.

\begin{thm}\label{thm:stable-lim-III}
The following are equivalent for a derivator \D.
\begin{enumerate}
\item The derivator \D is stable.\label{item:sl1}
\item Homotopy finite colimits and homotopy finite limits commute in \D.\label{item:sl2}
\newitemgroup Left homotopy finite left Kan extensions and arbitrary right Kan extensions commute in \D.\label{item:sl3a}
\itemgroup Arbitrary left Kan extensions and right homotopy finite right Kan extensions commute in \D.\label{item:sl3b}
\newitemgroup Every left exact morphism $\D^A\to\D^B,A,B\in\cCat,$ preserves left homotopy finite left Kan extensions.\label{item:sl4a}
\itemgroup Every right exact morphism $\D^A\to\D^B,A,B\in\cCat,$ preserves right homotopy finite right Kan extensions.\label{item:sl4b}
\stopitemgroup
\item The derivator \D is pointed and $C\colon\D^{[1]}\to\D$ preserves right homotopy finite right Kan extensions.\label{item:sl5}
\item The derivator \D is pointed and $C\colon\D^{[1]}\to\D$ preserves homotopy finite limits.\label{item:sl6}
\item The derivator \D is pointed and $C\colon\D^{[1]}\to\D$ preserves $F$.\label{item:sl7}
\end{enumerate}
\end{thm}
\begin{proof}
If \D is stable, then also the shifted derivators $\D^A,A\in\cCat,$ are stable \cite[Prop.~4.3]{groth:ptstab}. Consequently, every left exact morphism $\D^A\to\D^B$ is also right exact \cite[Prop.~9.8]{groth:can-can} and it hence preserves left homotopy finite left Kan extensions \cite[Thm.~9.14]{groth:can-can}. This and a dual argument shows that statement~\ref{item:sl1} implies statements~\ref{item:sl4a} and~\ref{item:sl4b}. Since right Kan extension morphisms are right adjoint morphisms and hence left exact, the implications~\ref{item:sl4a} implies~\ref{item:sl3a} and~\ref{item:sl3a} implies~\ref{item:sl2} are immediate. Moreover,~\ref{item:sl2} implies~\ref{item:sl1} by \autoref{thm:stable-lim-II}, and, by duality, it remains to incorporate the three final statements. Statement~\ref{item:sl1} implies statement~\ref{item:sl5} since $C$ is left exact in this case and it hence preserves right homotopy finite right homotopy Kan extensions \cite[Thm.~9.14]{groth:can-can}. The implications~\ref{item:sl5} implies~\ref{item:sl6} and \ref{item:sl6} implies~\ref{item:sl7} being trivial, it remains to show that~\ref{item:sl7} implies~\ref{item:sl1} which is already taken care of by the proof of \autoref{thm:stable-lim-II}.
\end{proof}

There are, of course, various additional minor variants of the characterizations in \autoref{thm:stable-lim-III} obtained, for example, by replacing $C$ by $\cof\colon\D^{[1]}\to\D^{[1]}$. 

\begin{rmk}\label{rmk:interpretation}
A typical slogan is that spectra are obtained from pointed topological spaces if one forces the suspension to become an equivalence. This slogan is made precise by \autoref{thm:stable-known} and the fact that the derivator of spectra is the stabilization of the derivator of pointed topological spaces \cite{heller:stable}. \autoref{thm:stable-known} and \autoref{thm:stable-lim-III} make precise various additional slogans saying, for instance, that spectra are obtained from spaces or pointed spaces by forcing certain colimit and limit type constructions to commute. We illustrate this by two examples.
\begin{enumerate}
\item If one forces homotopy finite colimits and homotopy finite limits to commute in the derivator of spaces, then one obtains the derivator of spectra. 
\item If one forces partial cones and partial fibers of squares to commute in the derivator of pointed spaces, then this yields the derivator of spectra.
\end{enumerate}
\end{rmk}

\begin{rmk}\label{rmk:stable-rep-triv}
The phenomenon that certain colimits and limits commute is well-known from ordinary category theory. To mention an instance, let us recall that filtered colimits are exact in Grothendieck abelian categories, i.e., filtered colimits and finite limits commute in such categories. Additional such statements hold in locally presentable categories, Grothendieck topoi, and algebraic categories. 

Now, the phenomenon of stability is invisible to ordinary category theory; in fact, a represented derivator is stable if and only if the representing category is trivial (this follows from \autoref{thm:stable-known} since the suspension morphism is trivial in pointed represented derivators). As a consequence the commutativity statements in \autoref{thm:stable-lim-III} have no counterparts in ordinary category theory.
\end{rmk}

\section{Stability versus absoluteness}
\label{sec:galois}

The close family resemblance between \cref{prop:ptd-comm,thm:stable-lim-III} suggests the following definition.

\begin{defn}
  Let $\Phi$ be a class of functors between small categories.
  A derivator \D is \textbf{left $\Phi$-stable} if for every $(u\colon A\to B)\in\Phi$, left Kan extensions along $u$ in $\D$ commute with arbitrary right Kan extensions.
  Dually, $\D$ is \textbf{right $\Phi$-stable} if $\D\op$ is left $\Phi$-stable, i.e.\ right Kan extensions along each $u\in \Phi$ in $\D$ commute with arbitrary left Kan extensions.
  If $\D$ is left (resp.\ right) $\Phi$-stable, we say that $\Phi$ is left (resp.\ right) \textbf{\D-absolute}.
\end{defn}

We may take $\Phi$ to be a class of categories instead of functors, in which case we identify a category $A$ with the unique functor $A\to\bbone$.
In the next section we will show that left $\Phi$-stability coincides with right $\Phi\op$-stability.

\begin{egs}
  A derivator \D is pointed if and only if it is left $\emptyset$-stable, if and only if it is right $\emptyset$-stable, and if and only if it is left stable for the class of cosieves, if and only if it is right stable for the class of sieves.
  Similarly, \D is stable if and only if it is left stable for the class of homotopy finite categories, if and only if it is right stable for the same class, if and only if it is left stable for the class of left homotopy finite functors, if and only if it is right stable for the class of right homotopy finite functors.
\end{egs}

This notion of relatively stable derivators allows us to construct the following Galois correspondence.

\begin{con}
  Given a class $\Phi$ of functors between small categories, we write $\stabl(\Phi)$ for the collection of left $\Phi$-stable derivators.
  Dually, given a collection $\Upsilon$ of derivators, we write $\absl(\Upsilon)$ for the class of left $\Upsilon$-absolute functors (i.e.\ functors that are left \D-absolute for all $\D\in\Upsilon$).
  Then $\stabl$ and $\absl$ are a Galois correspondence (a contravariant adjunction of partial orders) between the classes of functors and collections of derivators.
  In particular, we have
  \[ \Phi \subseteq \absl(\Upsilon) \iff \Upsilon \subseteq \stabl(\Phi) \]
  and
  \[ \Phi \subseteq \absl(\stabl(\Phi)) \qquad \Upsilon \subseteq \stabl(\absl(\Upsilon)) \]
  \[ \stabl(\Phi) = \stabl(\absl(\stabl(\Phi))) \qquad \absl(\Upsilon) = \absl(\stabl(\absl(\Upsilon))). \]
  Dually, we have $\stabr$ and $\absr$.
\end{con}

\begin{egs}
  \cref{prop:ptd-comm} can now be restated by saying that $\stabl(\{\emptyset\})$ and $\stabr(\{\emptyset\})$ are the collection $\mathsf{POINT}$ of pointed derivators, while $\absl(\mathsf{POINT})$ contains all cosieves and $\absr(\mathsf{POINT})$ contains all sieves.
  In particular, $\mathsf{POINT}$ is a fixed point of both Galois correspondences.
  Similarly, \cref{thm:stable-lim-III} can be restated by saying that $\stabl(\mathsf{FIN})$ and $\stabr(\mathsf{FIN})$, for $\mathsf{FIN}$ the class of homotopy finite categories, are both the collection $\mathsf{STABLE}$ of stable derivators; while $\absl(\mathsf{STABLE})$ contains all left homotopy finite functors and $\absr(\mathsf{STABLE})$ contains all right homotopy finite functors.

  The cone functor $C\colon\D^{[1]} \to \D$ is not a colimit (though it is a weighted colimit, in the sense to be defined in \cref{con:wcolim}, for a suitable enrichment), so we cannot consider ``$\stabl(\{C\})$''.
  However, if the pushout functor $\D^{\ulcorner} \to \D$ is continuous, then so is $C$, since $C$ is the composite of a pushout, a right Kan extension, and an evaluation morphism.
  Thus, we can say that $\mathsf{STABLE} = \stabl(\{\emptyset,\ulcorner\})$ and similarly $\mathsf{STABLE} = \stabr(\{\emptyset,\lrcorner\})$.
\end{egs}

\begin{eg}
  Of course, $\stabl(\emptyset)$ and $\stabr(\emptyset)$ are the collection $\mathsf{DERIV}$ of all derivators, while $\absl(\emptyset)$ and $\absr(\emptyset)$ are the class $\mathsf{FUNC}$ of all functors.
  However, $\absl(\mathsf{DERIV})$ and $\absr(\mathsf{DERIV})$ are nonempty; for instance, $\absl(\mathsf{DERIV})$ contains all left adjoint functors, $\absr(\mathsf{DERIV})$ all right adjoint functors, and they both include the splitting of idempotents.
  On the other hand, $\stabl(\mathsf{FUNC})$ and $\stabr(\mathsf{FUNC})$ include only the trivial derivator, by~\cite[Remark 9.4]{ps:linearity}.
\end{eg}

\begin{eg}
  Let $\Phi=\mathsf{FINDISC}$ be the class of finite discrete categories.
  Since $\emptyset\in\mathsf{FINDISC}$, any left or right $\Phi$-stable derivator is pointed.
  It is easy to see that $\stabl(\mathsf{FINDISC}) = \stabl(\{\emptyset,2\})$, where $2$ denotes the discrete category with two objects, and similarly for $\stabr$.

  In fact, we have $\stabl(\mathsf{FINDISC}) = \stabr(\mathsf{FINDISC}) = \mathsf{SEMIADD}$, the collection of semiadditive derivators.
  For since $\D^2 \simeq\D\times\D$ by one of the derivator axioms, the left and right Kan extensions along $2\to \bbone$ are just binary coproducts and products.
  Then if $\D$ is pointed and binary coproducts preserve all limits, then in particular they preserve binary products, which means that
  \[ (X\times Z) + (Y\times W) \cong (X+Y)\times (Z+W) \]
  canonically.
  Taking $Y=Z=0$, we see that $X+W \cong X\times W$ canonically, so that $\D$ is semiadditive.
  Conversely, if $\D$ is semiadditive, then the coproduct and product functors $\D\times \D\to\D$ coincide, and in particular the coproduct is a right adjoint and so preserves all limits.
  Thus $\D$ is left $\mathsf{FINDISC}$-stable if and only if it is semiadditive, and dually for right $\mathsf{FINDISC}$-stability.
\end{eg}

There are a number of natural questions suggested by this phrasing of the characterization theorems:
\begin{enumerate}
\item By definition, $\D$ is left $u$-stable if and only if $u_!\colon \D^A\to \D^B$ is continuous.
  But a continuous functor is crying out to be a right adjoint, for instance if there is an adjoint functor theorem.
  General derivators have no adjoint functor theorem, but does $u_!$ happen to be a right adjoint anyway?
\item \cref{prop:ptd-comm,thm:stable-lim-III} are self-dual, and in particular $\mathsf{POINT}$ and $\mathsf{STABLE}$ are fixed points of both Galois connections.
  Is there an abstract explanation for this?
\item We have seen that interesting collections of derivators like $\mathsf{POINT}$ and $\mathsf{STAB}$ can be generated as $\stabl(\Phi)$ for very small classes $\Phi$ of functors such as $\{\emptyset\}$ and $\{\emptyset,\ulcorner\}$.
  Can they also be generated as $\stabl(\absl(\Upsilon))$ for ``manageable'' collections $\Upsilon$ of derivators?
  For instance, are there ``universal'' pointed or stable derivators that suffice to detect whether a given functor is absolute for all pointed or stable derivators?
\end{enumerate}

To attack these questions, we use the technology of enriched derivators and weighted limits.
We will see that it suffices to answer the first two questions positively, but it is not quite adequate for the third in general, although in particular cases the answer is yes.
In~\cite{gs:enriched} we will use a better technology to answer the third question positively in general as well.

\section{Enriched derivators}
\label{sec:enriched-derivators}

We begin by defining the basic notions of enriched derivators.  We freely make use of the language and techniques established in \cite{gps:additivity}, in particular the language of \emph{monoidal derivators} as it is developed in detail in \cite[\S3]{gps:additivity}. In that paper there is also a detailed discussion of \emph{two-variable adjunctions of derivators} \cite[\S\S8-9]{gps:additivity}.

A monoidal derivator \V is a pseudo-monoid object in $\cDER$ (the 2-category of derivators and pseudonatural transformations) such that the monoidal structure $\otimes\colon\V\times\V\to\V$ preserves colimits separately in both variables. The pseudo-monoid structure precisely amounts to a lift of $\V\colon\cCat\op\to\cCAT$ against the forgetful functor from the $2$-category of monoidal categories, strong monoidal functors, and monoidal transformations. The resulting monoidal structures are denoted by $(\V(A),\otimes_A,\lS_A)$.

We will also have occasion to consider the following weaker notions.

\begin{defn}
  A \textbf{left derivator} is a prederivator satisfying all the axioms of a derivator except the existence of right Kan extensions.
\end{defn}

A morphism of left derivators is again a pseudonatural transformation, giving a 2-category \cLDER.
We can define two-variable morphisms of left derivators, and (separate) preservation of colimits, just as for derivators.

\begin{defn}
  A \textbf{monoidal left derivator} is a left derivator with a pseudo-monoid structure that preserves colimits separately in both variables.
  If \V is a monoidal left derivator, a \textbf{\V-module} is a cocontinuous pseudo-module, i.e.\ a left derivator \D with an action $\V\times\D\to \D$ that is coherently associative and unital and preserves colimits separately in both variables.
  We say that \D is a \textbf{\V-opmodule} if $\D\op$ is a \V-module.
  A \textbf{closed \V-module}, or \textbf{\V-enriched derivator}, is a \V-module whose action is part of a two-variable adjunction (hence, in particular, it is also a \V-opmodule).
\end{defn}

Now recall that derivator morphisms of two variables come in three different forms; see \cite[\S3 and \S5]{gps:additivity}. We right away specialize to the situation of an action as above.
\begin{enumerate}
\item In the \emph{internal form} $\otimes_A\colon\V(A)\times\D(A)\to\D(A)$ which naively is given by $(W\otimes_A X)_a=W_a\otimes X_a$, where $\otimes\colon\V(\bbone)\times\D(\bbone)\to\D(\bbone)$ denotes the underlying functor of two variables.
\item In the \emph{external form} $\otimes\colon\V(A)\times\D(B)\to\D(A\times B)$, which we think of as being defined by $(W\otimes X)_{a,b}=W_a\otimes X_b$.
\item Finally, in the \emph{canceling form} $\otimes_{[A]}\colon\V(A\op)\times\D(A)\to\D(\bbone)$ which is obtained from the external form by composing it with the coend functor 
\[
\int^A\colon\D(A\op\times A)\to\D(\bbone).
\]
For the notion of (co)ends in derivators we refer to \cite[\S5 and Appendix~A]{gps:additivity}.
\end{enumerate}
Note the different notation used for these three variants; the notation for internal versions was already used for the monoidal categories $(\V(A),\otimes_A,\lS_A)$.


\begin{eg}
Every monoidal left derivator is, of course, a module over itself.
If it is a closed module over itself, we call it a \textbf{closed monoidal left derivator}.

More generally, if \V is a monoidal left derivator, then any shift $\V^A$ is also a \V-module.
\end{eg}

We also have the following universal construction:

\begin{con}
  For any left derivators $\D,\E$, define $\pdh(\D,\E)$ by
  \[\pdh(\D,\E)(A) = \cDER(\D,\E^A)\]
  where a functor $u:A\to B$ induces the restriction functor
  \[ \pdh(\D,\E)(B) = \cDER(\D,\E^B) \to \cDER(\D,\E^A) = \pdh(\D,\E)(A) \]
  by postcomposition with $u^* \colon \E^B \to \E^A$.
  This makes $\pdh(\D,\E)$ into a left derivator, and indeed a derivator if \E is one; its Kan extension functors are also simply given by postcomposition.
  In this way \cLDER becomes a cartesian closed 2-category in an appropriate weak sense.
  In particular, $\pdh(\D,\D)$ is a pseudo-monoid under composition, and there is a canonical action $\pdh(\D,\D) \times \D \to \D$.
  However, this monoidal structure and action do not preserve colimits in the right variable, hence do not make \D into a $\pdh(\D,\D)$-module.

  Thus, we define a new left derivator $\ldh(\D,\E)$, for which $\ldh(\D,\E)(A)$ is the category of \emph{cocontinuous} morphisms $\D\to\E^A$.
  Since restriction and left Kan extension are cocontinuous morphisms, this is again a left derivator.
  The endomorphism object $\ldh(\D,\D)$, which we denote $\lend(\D)$, \emph{is} a monoidal left derivator under composition, and its action $\lend(\D)\times\D\to\D$ does make \D into an $\lend(\D)$-module.

  Explicitly, the external monoidal product of $F\colon\D\to\D^A$ and $G\colon\D\to\D^B$ is the morphism $GF\colon\D\to\D^{A\times B}$ whose component $\D(C) \to \D(C\times A\times B)$ is the composite $\D(C) \xto{F^C} \D(C\times A) \xto{G^{C\times A}} \D(C\times A\times B)$.
  Similarly, the external action of $F\colon\D\to\D^A$ on $X\in \D(B)$ is the image of $X$ under $F^B \colon\D(B) \to \D(B\times A)$.
  Both of these preserve colimits in both variables, on the left because colimits there are defined by postcomposition, and on the right because $F$ and $G$ preserve colimits.

  This construction is universal in the sense that if \V is a monoidal left derivator, then to make \D into a \V-module is equivalent to giving a cocontinuous monoidal morphism $\V\to\lend(\D)$.
  Specifically, the latter assigns to each $X\in \V(A)$ a morphism $\D\to \D^A$, which is the external tensor product with $X$.
  Monoidality of the morphism $\V\to\lend(\D)$ gives the associativity and unitality of the action, while its cocontinuity gives left cocontinuity of the action; right cocontinuity of the action comes from the fact that this morphism lands in $\lend(\D) = \ldh(\D,\D)$ rather than $\pdh(\D,\D)$.
\end{con}

Note that unlike $\pdh(\D,\E)$, the left derivator $\ldh(\D,\E)$ is not a derivator even if \E is: since limits and colimits do not in general commute, the limit in $\pdh(\D,\E)$ of cocontinuous morphisms need no longer be cocontinuous.
However, we can say;

\begin{lem}\label{thm:ldh-ran}
  If $u\colon A\to B$ is such that \E has right Kan extensions along $u$ that commute with arbitrary left Kan extensions, then so does $\ldh(\D,\E)$.
\end{lem}
\begin{proof}
  Right Kan extensions in $\pdh(\D,\E)$ are defined by postcomposition; if $u_*$ is cocontinuous then $\ldh(\D,\E)$ is closed under such postcomposition.
  Since left Kan extensions are also defined by postcomposition, commutativity follows.
\end{proof}

We now introduce the notion of weighted colimits.
First note that the internal, external, and canceling versions of morphisms of two-variables can be combined. In particular, given a monoidal derivator \V and $A,B,C\in\cCat$, there is the \textbf{(homotopy) tensor product of functors}
\[
\otimes_{[B]}\colon \V(A\times B\op)\times\V(B\times C\op)\stackrel{\otimes}{\to}\V(A\times B\op\times B\times C\op)\stackrel{\int^B}{\to}\V(A\times C\op),
\]
and also this operation enjoys associativity and unitality properties.

\begin{thm}[{\cite[Theorem~5.9]{gps:additivity}}]\label{thm:bicategory}
  If \V is a monoidal left derivator, then there is a bicategory $\cProf(\V)$ described as follows:
  \begin{itemize}
  \item Its objects are small categories.
  \item Its hom-category from $A$ to $B$ is $\V(A\times B\op)$.
  \item Its composition functors are the external-canceling tensor products
    \[ \otimes_{[B]} \colon \V(A\times B\op) \times \V(B\times C\op) \too \V(A\times C\op). \]
  \item The identity 1-cell of a small category $B$ is
    \begin{equation}
      \lI_B\;=\;(t,s)_! \lS_{\tw(B)} \;\cong\; (t,s)_! \pi_{\tw(B)}^* \lS_{\bbone} \; \in \V(B\times B\op).\label{eq:unit}
    \end{equation}
  \end{itemize}
\end{thm}

The notation related to the identity $1$-cells $\lI_B\in\V(B\times B\op)$, also called \textbf{identity profunctors}, is as follows. $\tw(B)$ is the \textbf{twisted arrow category} of $B$, i.e., the category of elements of $\hom_B$, and the functor $(t,s)\colon\tw(B)\to B\times B\op$ sends a morphism to its target and source (see \cite[\S5]{gps:additivity}). We refer to $\cProf(\V)$ as the \textbf{bicategory of profunctors} in \V.

\begin{con}\label{con:wcolim}
Let \V be a monoidal left derivator and let \D be a \V-module with tensors $\otimes\colon\V\times\D\to\D$. The external-canceling version of this morphism yields functors
\[
\otimes_{[B]}\colon\V(A\times B\op)\times \D(B\times C\op)\to\D(A\times C\op).
\]
Passing to parametrized versions of these functors, we obtain an external-canceling tensor morphism 
\[
\otimes_{[B]}\colon\V^{A\times B\op}\times\D^{B\times C\op}\to\D^{A\times C\op}.
\]
In particular, plugging in a fixed $W\in\V(A\times B\op)$ and specializing to $C=\bbone$, we obtain an induced partial morphism
\[
\colim^W=(W\otimes_{[B]}-)\colon\D^B\to\D^A,
\] 
the \textbf{weighted colimit morphism with weight} $W\in\V(A\times B\op)$. We abuse terminology and refer to a morphism as a weighted colimit if it is naturally isomorphic to $\colim^W$ for some $W$. In a dual way, if $\D$ is a \V-opmodule, one defines \textbf{weighted limits}
\[
\mathrm{lim}^W=(-\lhd_{[A]}W)\colon\D^A\to\D^B,
\] 
Moreover, if \D is a closed \V-module, then weighted colimits and weighted limits are always adjoint to each other:
\begin{equation}\label{eq:wcolim-adj}
(\colim^W,\mathrm{lim}^W)\colon\D^B\rightleftarrows\D^A.
\end{equation}
\end{con}

\begin{lem}\label{lem:wcolim}
Let \V be a monoidal left derivator and let \D be a \V-module.
\begin{enumerate}
\item The morphism $\otimes_{[B]}\colon\V^{A\times B\op}\times\D^{B\times C\op}\to\D^{A\times C\op}$ preserves colimits in each variable separately.
  In particular, any weighted colimit functor is cocontinuous.\label{item:wc1}
\item If \V is a monoidal derivator, and \D is a derivator and a closed \V-module, then $\otimes_{[B]}$ is a left adjoint of two variables.\label{item:wc2}
\item The morphism $(\lI_B\otimes_{[B]}-)\colon\D^B\to\D^B$ is naturally isomorphic to the identity morphism.\label{item:wc3}
\end{enumerate}
\end{lem}
\begin{proof}
  Statements~\ref{item:wc1} and~\ref{item:wc2} are true for the external-canceling variant of any cocontinuous two-variable morphism, while~\ref{item:wc3} follows from the same argument used to prove unitality of the bicategory $\cProf(\V)$.
\end{proof}

\begin{thm}\label{thm:wcolim}
  Let \V be a monoidal left derivator and let \D be a \V-module.
  Then:
  \begin{enumerate}
  \item Restriction morphisms $u^\ast\colon\D^B\to\D^A$ are \V-weighted colimits.\label{item:wcl1}
  \item Left Kan extension morphisms $u_!\colon\D^A\to\D^B$ are \V-weighted colimits.\label{item:wcl2}
  \item If \V and \D are pointed derivators, then right Kan extension morphisms $u_\ast\colon\D^A\to\D^B$ along sieves are weighted colimits.\label{item:wcl3}
  \item If \V and \D are stable derivators, then right homotopy finite right Kan extension morphisms $u_\ast\colon\D^A\to\D^B$ are weighted colimits.\label{item:wcl4}
\end{enumerate}
\end{thm}
\begin{proof}
For every fixed $X\in\D(B)$ and $u\colon A\to B$, pseudo-naturality of the partial morphism $(-\otimes_{[B]}X)\colon\V^{B\op}\to \D$ and \autoref{lem:wcolim} yields
\[
u^\ast(X)\cong u^\ast(\lI_B\otimes_{[B]}X)\cong\big((u\times\id)^\ast\lI_B\big)\otimes_{[B]}X.
\]
This defines a natural isomorphism $u^\ast\cong \big((u\times\id)^\ast\lI_B\big)\otimes_{[B]}-\colon\D^B\to\D^A$, thereby exhibiting $u^\ast$ as a weighted colimit.

Similarly, if we fix $X\in\D(A)$, then by \autoref{lem:wcolim} the partial morphism
\begin{equation}\label{eq:par-mor-wcolim}
(-\otimes_{[A]}X)\colon\V^{A\op}\to\D
\end{equation}
is cocontinuous.  Given a functor $u\colon A\to B$ we obtain natural isomorphisms
\[
u_!(X)\cong u_!(\lI_A\otimes_{[A]}X)\cong \big((u\times\id)_!\lI_A\big)\otimes_{[A]}X,
\]
hence a natural isomorphism $u_!\cong\big(\big((u\times\id)_!\lI_A\big)\otimes_{[A]}-\big)\colon\D^A\to\D^B$, identifying $u_!$ as a weighted colimit.

If \V and \D are pointed, then \eqref{eq:par-mor-wcolim} is a cocontinuous morphism of pointed derivators and hence automatically pointed, hence preserves right Kan extensions along sieves \cite[Cor.~8.2]{groth:can-can}.
Thus, a similar calculation as above yields for every such $u\colon A\to B$ a natural isomorphism
\[
u_\ast\cong \big(\big((u\times\id)_\ast\lI_A\big)\otimes_{[A]}-\big) \colon\D^A\to\D^B,
\]
exhibiting $u_\ast$ as a weighted colimit. 
Similarly, if \V and \D are stable derivators, we note that \eqref{eq:par-mor-wcolim} is an exact morphism of stable derivators (by \autoref{lem:wcolim} and \cite[Cor.~9.9]{groth:can-can}) and it hence preserves right homotopy finite right Kan extensions \cite[Thm.~9.14]{groth:can-can}.  
\end{proof}

Applying \cref{thm:wcolim} to the \V-module $\V^{C\op}$, we find that for any $X\in \V(B\times C\op)$ and $Y\in \V(A\times C\op)$ we have
\begin{align*}
  \big((u\times\id)^\ast\lI_B\big) \otimes_{[B]} X &\cong (u\times\id)^\ast X\\
  \big((u\times\id)_!\lI_A\big) \otimes_{[A]} Y &\cong (u\times\id)_! Y
\end{align*}
Note that in this case, $\otimes_{[B]}$ and $\otimes_{[A]}$ are the composition in $\cProf(\V)$; thus restriction and left Kan extension in \V can both be described using composition in $\cProf(\V)$.
The special objects $(u\times\id)^\ast\lI_B$ and $(u\times\id)_!\lI_A$ are sometimes called \textbf{base change objects}.
Dually, for any $X\in \V(E\times B\op)$ and $Y\in \V(E\times A\op)$ we have
\begin{align*}
  X \otimes_{[B]} \big((\id\times u\op)^\ast\lI_B\big) &\cong (\id\times u\op)^\ast X\\
  Y \otimes_{[A]} \big((\id\times u\op)_!\lI_A\big)  &\cong (\id\times u\op)_! Y
\end{align*}
In fact, these dual base change objects are actually isomorphic to the first two swapped:
\begin{align*}
  (\id\times u\op)^\ast\lI_B &\cong (u\times\id)_!\lI_A\\
  (\id\times u\op)_!\lI_A &\cong (u\times\id)^\ast\lI_B
\end{align*}
This all follows from the fact that $\cProf(\V)$ is actually a ``framed bicategory''; see~\cite{shulman:frbi} and~\cite[(15.2)]{ps:linearity}.

\begin{rmk}
Let \V be a monoidal left derivator and \D a \V-module. For $u\colon A\to B$ in $\cCat$ we obtain an isomorphism
\[
u_!\cong ((\id\times u\op)^\ast\lI_B)\otimes_{[A]}-\colon\D^A\to\D^B.
\]
Specializing to $u=\pi_A\colon A\to\bbone$ we deduce that colimits are weighted colimits with constant weight $\pi_{A\op}^\ast\lS_\bbone$. More generally, the weight for $u_!$ has components
\[
((\id\times u\op)^\ast\lI_B)_{b,a}\cong\coprod_{\hom_B(ua,b)}\lS_\bbone,
\]
and the isomorphism $u_!X\cong((\id\times u\op)^\ast\lI_B)\otimes_{[A]}X$ is hence a left derivator version of the usual coend formula for left Kan extensions in sufficiently cocomplete categories (\cite[Thm.~X.4.1]{maclane}). 
\end{rmk}

\section{Stability via weighted colimits}
\label{sec:stab-via-wcolim}

\cref{thm:wcolim}\ref{item:wcl3} and~\ref{item:wcl4} cry out for a generalization to $\Phi$-stability.

\begin{defn}
  If $\Phi$ is a class of functors $u\colon A\to B$, we define a left derivator \D to be \textbf{right $\Phi$-stable} if it \emph{has} right Kan extensions along each $u\in \Phi$ which moreover commute with arbitrary left Kan extensions.
\end{defn}

By \cref{thm:ldh-ran}, if $\D$ is right $\Phi$-stable, then so is $\lend(\D)$.

\begin{thm}\label{thm:stable-dual}
  Let \V be a monoidal left derivator and $u\colon A\to B$ a functor.
  The following are equivalent:
  \begin{enumerate}
  \item $\V$ is right $u\op$-stable.\label{item:sd0}
  \item The base change profunctor $(u\times\id)_!\lI_A\in \cProf(\V)(B,A)$ has a right adjoint in $\cProf(\V)$.\label{item:sd1}
  \item The base change profunctor $(\id\times u)_!\lI_{A\op}\in \cProf(\V)(A\op,B\op)$ has a left adjoint in $\cProf(\V)$.\label{item:sd1op}
  \item The morphism $u_!\colon\V^A\to\V^B$ has a left adjoint that is a weighted colimit functor.\label{item:sd2}
  \item The right Kan extension $(u\op)_\ast\colon \V^{A\op} \to \V^{B\op}$ exists and is a \V-weighted \emph{colimit} functor.\label{item:sd2op}
  \end{enumerate}
\end{thm}
\begin{proof}
  We first show that~\ref{item:sd1} and~\ref{item:sd1op} are equivalent.
  The right adjoint in~\ref{item:sd1} would be an object $Z\in \V(A\times B\op)$, whereas the left adjoint in~\ref{item:sd1op} would be an object $Z'\in \V(B\op\times (A\op)\op)$; but of course these are equivalent categories.
  The unit and counit in~\ref{item:sd1} would be morphisms
  \begin{align*}
    \eta &: \lI_B \to (u\times\id)_!\lI_A \otimes_{[A]} Z \cong (u\times\id)_! Z\\
    \ep &: (\id\times u\op)^\ast Z \cong  Z \otimes_{[B]} (u\times\id)_!\lI_A \to \lI_A
  \end{align*}
  whereas the unit and counit in~\ref{item:sd1op} would be morphisms
  \begin{align*}
    \eta' &: \lI_{B\op} \to Z'\otimes_{[A\op]} (\id\times u)_!\lI_{A\op} \cong (\id\times u)_! Z'\\
    \ep' &: (u\op\times \id)^\ast Z' \cong (\id\times u)_!\lI_{A\op} \otimes_{[B\op]} Z' \to \lI_{A\op}.
  \end{align*}
  Thus, to give $\eta$ is the same as to give $\eta'$, since $\lI_{B\op}$ corresponds to $\lI_B$ under the equivalence $\V(B\times B\op) \simeq \V(B\op \times (B\op)\op)$, and so on.
  We leave it to the reader to check that the triangle identities likewise correspond.

  Now we show that~\ref{item:sd0} implies~\ref{item:sd1}.
  We take the right adjoint to be $(\id\times u\op)_\ast \lI_A \in \cProf(\V)(A,B)$.
  Then morphisms $Y\to  (\id\times u\op)_\ast \lI_A$ are equivalent to morphisms $(\id\times u\op)^\ast Y \to \lI_A$, i.e.\ morphisms $Y\otimes_{[B]} (u\times\id)_!\lI_A \to \lI_A$.
  In bicategorical language, $(\id\times u\op)_\ast \lI_A$ is a \emph{right lifting} of $\lI_A$ along $(u\times\id)_!\lI_A$.
  In general, a right lifting of the identity along a 1-cell $X$ is a right adjoint as soon as it is preserved by precomposition with $X$ (see for instance~\cite[Theorem X.7.2]{maclane} or~\cite[16.4.12]{maysig:pht}).
  In our case when $X = (u\times\id)_!\lI_A$, precomposition with $X$ is just left Kan extension along $u$, which by our assumption of $u\op$-stability preserves the right Kan extension $(\id\times u\op)_\ast$.
  Thus, $(u\times\id)_!\lI_A \otimes_{[A]} (\id\times u\op)_\ast \lI_A \cong (\id\times u\op)_\ast \big((u\times\id)_!\lI_A\big)$, so it has an analogous universal property, as desired.

  Now, if~\ref{item:sd1} holds, then since weighted colimits are contravariantly functorial on profunctors, the adjunction $(u\times\id)_!\lI_A \adj Z$ yields an adjunction $\colim^Z \adj \colim^{(u\times\id)_!\lI_A} = u_!$.
  This gives~\ref{item:sd2}.
  Conversely, if $Z\in\V(A\times B\op)$ is such that $\colim^Z \adj u_! = \colim^{(u\times\id)_!\lI_A}$, then since composition in $\cProf(\V)$ is a special case of weighted colimits, we have natural adjunctions $(Z\otimes_{[B]} -) \adj ((u\times\id)_!\lI_A \otimes_{[A]} -)$, which by the bicategorical Yoneda lemma induce an adjunction $(u\times\id)_!\lI_A\adj Z$ in $\cProf(\V)$.

  Similarly,~\ref{item:sd2op} is equivalent to~\ref{item:sd1op}, since $\colim^{(\id\times u)_!\lI_{A\op}} \cong (u\op)^\ast$.
  Finally, if~\ref{item:sd2op} holds then $(u\op)_\ast$, being a weighted colimit, commutes with all left Kan extensions, so that $\V$ is right $u\op$-stable.
\end{proof}

Note that \cref{thm:stable-dual}\ref{item:sd2op} is a generalization of \cref{thm:wcolim}\ref{item:wcl3} and~\ref{item:wcl4}.
This can be regarded as an explanation of ``why'' $\Phi$-limits in a right $\Phi$-stable derivator commute with all colimits: they are themselves weighted colimits.
(If \V is not symmetric, then arbitrary weighted colimits need not commute with arbitrary other \emph{weighted} colimits.
However, left Kan extensions always commute with all weighted colimits, by \cref{lem:wcolim}\ref{item:wc1}.
If we express left Kan extensions as weighted colimits themselves, then they are in the ``center'' of \V.
If \V \emph{is} symmetric, then the duality $A\mapsto A\op$ extends to a self-duality of the bicategory $\cProf(\V)$, from which the equivalence of~\ref{item:sd1} and~\ref{item:sd1op} follows formally; the proof given above shows that this equivalence remains true even in the non-symmetric case, due to this ``centrality''.)

Now we can answer our first two questions from \cref{sec:galois}.

\begin{cor}\label{thm:stab-op}
  For a derivator $\D$ and a class of functors $\Phi$, the following are equivalent.
  \begin{enumerate}
  \item $\D$ is left $\Phi$-stable, i.e.\ $\Phi$-colimits in \D commute with arbitrary limits.\label{item:so1}
  \item For each $u\in\Phi$, the morphism $u_! :\D^A \to\D^B$ has a left adjoint.\label{item:so2}
  \item $\D$ is right $\Phi\op$-stable, i.e.\ $\Phi\op$-limits in \D commute with arbitrary colimits.\label{item:so3}
  \item For each $u\in\Phi\op$, the morphism $(u\op)_\ast :\D^{A\op} \to\D^{B\op}$ has a right adjoint.\label{item:so4}
  \end{enumerate}
\end{cor}
\begin{proof}
  We have~\ref{item:so2} implies~\ref{item:so1}, since all right Kan extensions exist in a derivator (as opposed to a left derivator), and are preserved by any right adjoint morphism.
  Dually,~\ref{item:so4} implies~\ref{item:so3}.
  We will prove that~\ref{item:so3} implies~\ref{item:so2}; by duality then also~\ref{item:so1} implies~\ref{item:so4} and we are done.
  If $\D$ is right $\Phi\op$-stable, then we remarked above that $\lend(\D)$ is right $\Phi\op$-stable, and $\D$ is a $\lend(\D)$-module.
  Therefore, by \cref{thm:stable-dual}, $u_!$ has a left adjoint (that is even a weighted colimit functor) for each $u\in\Phi$.
\end{proof}

\begin{cor}
  If $\Phi=\Phi\op$, then $\stabl(\Phi)=\stabr(\Phi)$.\qed
\end{cor}

This explains the self-dual nature of pointedness, semiadditivity, and stability as due to the fact that $\Phi=\{\emptyset\}$, $\Phi=\mathsf{FINDISC}$, and $\Phi=\mathsf{FIN}$ are self-dual.
Similarly, it explains the identity $\stabl(\{\emptyset,\ulcorner\}) = \stabr(\{\emptyset,\lrcorner\}) = \mathsf{STABLE}$, since $(\ulcorner)\op = \lrcorner$.

\section{Stability via iterated adjoints}
\label{sec:fun}

In particular, \cref{thm:stab-op} implies that we can characterize $\Phi$-stability in terms of iterated adjoints to constant morphism morphisms.  In this section we describe what this looks like more concretely in the pointed and stable cases.

\begin{prop}\label{prop:char-ptd}
The following are equivalent for a derivator \D.
\begin{enumerate}
\item The derivator \D is pointed.\label{item:p1}
\item The morphism $\emptyset_!\colon\D^\emptyset\to\D$ is a right adjoint.\label{item:p2}
\item The left Kan extension morphism $1_!\colon\D\to\D^{[1]}$ along the universal cosieve $1\colon\bbone\to[1]$ is a right adjoint.\label{item:p3}
\item For every cosieve $u\colon A\to B$ the left Kan extension morphism $u_!\colon\D^A\to\D^B$ is a right adjoint.\label{item:p4}
\item The morphism $\emptyset_\ast\colon\D^\emptyset\to\D$ is a left adjoint.\label{item:p5}
\item The right Kan extension morphism $0_\ast\colon\D\to\D^{[1]}$ along the universal sieve $0\colon\bbone\to[1]$ is a left adjoint.\label{item:p6}
\item For every sieve $u\colon A\to B$ the right Kan extension morphism $u_\ast\colon\D^A\to\D^B$ is a left adjoint.\label{item:p7}
\end{enumerate}
\end{prop}
\begin{proof}
By duality it suffices to show the equivalence of the first four statements. The implication~\ref{item:p1} implies~\ref{item:p4} is \cite[Cor.~3.8]{groth:ptstab}. Since the empty functor $\emptyset\colon\emptyset\to\bbone$ is a cosieve it remains to show that~\ref{item:p2} or~\ref{item:p3} imply~\ref{item:p1}. The case of~\ref{item:p2} is taken care of by the proof of \cite[Cor.~3.5]{groth:ptstab}. In the remaining case, if $1_!$ is a right adjoint it preserves all limits and hence terminal objects. Since the terminal object in $\D([1])$ looks like $(\ast\to\ast)$, this has by \cite[Prop.~1.23]{groth:ptstab} to be isomorphic to $1_!(\ast)\cong(\emptyset\to\ast)$. Evaluating this isomorphism at $0$ shows that \D is pointed.
\end{proof}

These additional adjoint functors are sometimes referred to as \textbf{(co)exceptional inverse image functors} (see \cite[\S3]{groth:ptstab}).

\begin{rmk}\label{rmk:C-inverse}
In \cite{groth:ptstab} the cone $C\colon\D^{[1]}\to\D$ and the fiber $F\colon\D^{[1]}\to\D$ is defined in pointed derivators only, but the same formulas make perfectly well sense in arbitrary derivators. It turns out that a derivator is pointed if and only if $C$ is a left adjoint if and only if $F$ is a right adjoint. In that case, there are adjunctions $C\dashv 1_!$ and $0_\ast\dashv F$, exhibiting $C$ and $F$ as (co)exceptional inverse image functors; see \cite[Prop.~3.22]{groth:ptstab}. 
\end{rmk}

\begin{rmk}
  In \cref{thm:stable-fun} we will characterize stable derivators with a simliar list of conditions, essentially by combining \cref{thm:stable-lim-III,thm:stab-op}.
  We could similarly have proven \cref{prop:char-ptd} by combining \cref{prop:ptd-comm,thm:stab-op}, but we chose instead to give a proof with a closer connection to previous literature.
\end{rmk}

Let \D be a derivator and let $1\colon\bbone\to[1]$ classify the terminal object $1\in[1]$. In every derivator \D there are Kan extension adjunctions $(1_!,1^\ast)\colon\D\rightleftarrows\D^{[1]}$
and $(1^\ast,1_\ast)\colon\D\rightleftarrows\D^{[1]}$, and we hence have an adjoint triple
\[
1_!\dashv 1^\ast\dashv 1_\ast.
\]
Similarly, associated to the functor $0\colon\bbone\to[1]$ there is the adjoint triple
\[
0_!\dashv 0^\ast\dashv 0_\ast.
\]

\begin{prop}\label{prop:univ-sieve}
Let \D be a derivator and let $0,1\colon\bbone\to[1]$ classify the objects $0,1\in[1]$.
\begin{enumerate}
\item The morphisms $0_!,1_\ast\colon\D\to\D^{[1]}$ are fully faithful and induce an equivalence onto the full subderivator spanned by the isomorphisms. This equivalence is pseudo-natural with respect to arbitrary morphisms of derivators.
\item There are canonical isomorphism $0_!\cong \pi_{[1]}^\ast\cong 1_\ast\colon\D\to\D^{[1]}$.
\end{enumerate}
\end{prop}
\begin{proof}
Both morphisms $0_!$ and $1_\ast$ are fully faithful and the essential image consists precisely of the isomorphisms by \cite[Prop.~3.12]{groth:ptstab}. Since derivators are invariant under equivalences of prederivators, the subprederivator of isomorphisms is a derivator. The equivalence is pseudo-natural with respect to arbitrary morphisms since all morphisms preserve left Kan extensions along left adjoint functors (see \cite[Prop.~5.7]{groth:can-can} and \cite[Rmk.~6.11]{groth:can-can}). As for the second statement, there is an adjoint triple $0\dashv\pi_{[1]}\dashv 1$ and hence an induced adjoint triple $1^\ast\dashv \pi_{[1]}^\ast\dashv 0^\ast$. This yields canonical isomorphisms $1_\ast\cong\pi_{[1]}^\ast$ and $0_!\cong\pi_{[1]}^\ast$. 
\end{proof}

We refer to $\pi_{[1]}^\ast\colon\D\to\D^{[1]}$ as the \textbf{constant morphism morphism}.

\begin{cor}
In every derivator \D there is an adjoint $5$-tuple
\begin{equation}\label{eq:5tuple}
1_!\dashv 1^\ast\dashv \pi_{[1]}^\ast\dashv 0^\ast\dashv 0_\ast.
\end{equation}
\end{cor}
\begin{proof}
This is immediate from \autoref{prop:univ-sieve}.
\end{proof}

\begin{prop}
A derivator \D is pointed if and only if the adjoint $5$-tuple \eqref{eq:5tuple} extends to an adjoint $7$-tuple, which is then given by
\begin{equation}\label{eq:7tuple}
C\dashv 1_!\dashv 1^\ast\dashv \pi_{[1]}^\ast\dashv 0^\ast\dashv 0_\ast\dashv F.
\end{equation}
\end{prop}
\begin{proof}
This is immediate from \autoref{prop:char-ptd} and \autoref{rmk:C-inverse}.
\end{proof}

\begin{rmk}
While in any pointed derivator there is by \cite[Prop.~3.20]{groth:ptstab} an adjunction
\[
(\cof,\fib)\colon\D^{[1]}\rightleftarrows\D^{[1]},
\]
in pointed derivators the morphism $C$ is the sixth left adjoint of $F$. 
\end{rmk}

\begin{thm}\label{thm:stable-fun}
The following are equivalent for a pointed derivator $\D$.
\begin{enumerate}
\item The derivator $\D$ is stable.\label{item:sf1}
\item The cone morphism $C\colon\D^{[1]}\to\D$ is a right adjoint.\label{item:sf2}
\item For any homotopy finite category $A$, the colimit morphism $\colim : \D^A \to \D$ is a right adjoint.\label{item:sf3}
\item For any left homotopy finite functor $u:A\to B$, the left Kan extension morphism $u_!: \D^A \to \D^B$ is a right adjoint.\label{item:sf4}
\item The fiber morphism $F\colon\D^{[1]}\to\D$ is a left adjoint.\label{item:sf2a}
\item For any homotopy finite category $A$, the limit morphism $\lim : \D^A \to \D$ is a left adjoint.\label{item:sf3a}
\item For any right homotopy finite functor $u:A\to B$, the right Kan extension morphism $u_*: \D^A \to \D^B$ is a left adjoint.\label{item:sf4a}
\item The adjoint $7$-tuple \eqref{eq:7tuple} extends to a doubly-infinite chain of adjoint morphisms.\label{item:sf5}
\end{enumerate}
\end{thm}
\begin{proof}
  Combining \cref{thm:stable-lim-III,thm:stab-op}, we see that~\ref{item:sf1} implies~\ref{item:sf4}, which clearly implies~\ref{item:sf3}, while~\ref{item:sf3} implies~\ref{item:sf2} since the cone is a composite of a right extension by zero with a pushout.
  And~\ref{item:sf2} implies~\ref{item:sf1} by \cref{thm:stable-lim-III}\ref{item:sl6}, since right adjoints preserve all limits, so the first four statements are equivalent.
  The equivalence of~\ref{item:sf1} with~\ref{item:sf2a}, \ref{item:sf3a}, and~\ref{item:sf4a} is dual.
  Evidently~\ref{item:sf5} implies~\ref{item:sf2}.
  And conversely, if \D is a stable derivator, then by \autoref{prop:stable-known-mod} there are natural isomorphisms
\[
\Sigma F\toiso C\qquad\text{and}\qquad F\toiso\Omega C.
\]
Since $\Sigma$ and $\Omega$ are equivalences in stable derivators (\autoref{thm:stable-known}), this shows that the outer morphisms in the adjoint $7$-tuple \eqref{eq:7tuple} match up to an equivalence. This implies that the adjoint $7$-tuple can be extended to a doubly-infinite chain of adjoint morphisms and that this chain has period six (in the obvious sense).
\end{proof}

We conclude by offering a first interpretation and visualization of this chain of morphisms.

\begin{rmk}
Let \D be a stable derivator. Then a few additional adjoint morphisms in the doubly-infinite sequence extending \eqref{eq:7tuple} are given by:
\[
\ldots\dashv\pi^\ast\Omega\dashv \Sigma 0^\ast\dashv 0_\ast\Omega\dashv C\dashv 1_!\dashv 1^\ast\dashv \pi^\ast\dashv 0^\ast\dashv 0_\ast\dashv F\dashv 1_!\Sigma\dashv \Omega 1^\ast\dashv \pi^\ast\Sigma\dashv\ldots
\]
In fact, this is immediate from the proof of \autoref{thm:stable-fun}.

In order to not get lost in all these morphisms, let us recall that Barratt--Puppe sequences in a stable derivator \D can be thought of as refinements of the more classical distinguished triangles. More precisely, associated to $(f\colon x\to y)\in\D^{[1]}$ there is the Barratt--Puppe sequence $BP(f)$ generated by $f$. This is a coherent diagram as in \autoref{fig:BP-sequence-stable} which vanishes on the boundary stripes and which makes all squares bicartesian.
\begin{figure}
\begin{equation}
\vcenter{
\xymatrix@-1pc{
\ar@{}[dr]|{\ddots}&\ar@{}[dr]|{\ddots}&\ar@{}[dr]|{\ddots}&&&&&&\\
\ar@{}[dr]|{\ddots}&\Omega Ff\ar[r]\ar[d]\pullbackcorner&\Omega x\ar[r]\ar[d]\pullbackcorner&0\ar[d]&&&&&\\
&0\ar[r]&\Omega y\ar[r]\ar[d]\pushoutcorner\pullbackcorner&Ff\ar[r]\ar[d]\pushoutcorner\pullbackcorner&0\ar[d]&&&&\\
&&0\ar[r]&x\ar[r]^-f\ar[d]\pushoutcorner\pullbackcorner&y\ar[r]\ar[d]\pushoutcorner\pullbackcorner&0\ar[d]&&&\\
&&&0\ar[r]&Cf\ar[r]\ar[d]\pushoutcorner\pullbackcorner&\Sigma x\ar[r]\ar[d]\pushoutcorner\pullbackcorner&0\ar[d]\ar@{}[rd]|{\ddots}&&\\
&&&&0\ar[r]\ar@{}[dr]|{\ddots}&\Sigma y\ar[r]\ar@{}[rd]|{\ddots}\pushoutcorner&\Sigma Cf\ar@{}[dr]|{\ddots}\pushoutcorner&\\
&&&&&&&&
}
}
\end{equation}
\caption{The Barratt--Puppe sequence of $f$}
\label{fig:BP-sequence-stable}
\end{figure}
(It turns out that $BP$ defines an equivalence of derivators (see \cite[Thm.~4.5]{gst:Dynkin-A}).)

Now, one half of the morphisms in the doubly-infinite chain simply amount to traveling in the Barratt--Puppe sequence in \autoref{fig:BP-sequence-stable}. If we imagine to sit on the morphism $f$ in $BP(f)$, then for every $n\geq 1$ an application of the $(2n\text{-}1)$-th left adjoint of $\pi^\ast$ to $f$ amounts to traveling $n$ steps in the positive direction. For low values this yields $y,Cf,\Sigma x,\Sigma y,$ and so on. There is a similar interpretation of the iterated right adjoints to $\pi^\ast$.

In order to obtain a similar visualization of the remaining adjoints, let us consider the Barratt--Puppe sequence $BP(\pi_{[1]}^\ast x), x\in\D$, of a constant morphism which then looks like \autoref{fig:BP-constant}.
\begin{figure}
\begin{equation}
\vcenter{
\xymatrix@-1pc{
\ar@{}[dr]|{\ddots}&\ar@{}[dr]|{\ddots}&\ar@{}[dr]|{\ddots}&&&&&&\\
\ar@{}[dr]|{\ddots}&0\ar[r]\ar[d]\pullbackcorner&\Omega x\ar[r]\ar[d]\pullbackcorner&0\ar[d]&&&&&\\
&0\ar[r]&\Omega x\ar[r]\ar[d]\pushoutcorner\pullbackcorner&0\ar[r]\ar[d]\pushoutcorner\pullbackcorner&0\ar[d]&&&&\\
&&0\ar[r]&x\ar[r]^-\id\ar[d]\pushoutcorner\pullbackcorner&x\ar[r]\ar[d]\pushoutcorner\pullbackcorner&0\ar[d]&&&\\
&&&0\ar[r]&0\ar[r]\ar[d]\pushoutcorner\pullbackcorner&\Sigma x\ar[r]\ar[d]\pushoutcorner\pullbackcorner&0\ar[d]\ar@{}[rd]|{\ddots}&&\\
&&&&0\ar[r]\ar@{}[dr]|{\ddots}&\Sigma x\ar[r]\ar@{}[rd]|{\ddots}\pushoutcorner&0\ar@{}[dr]|{\ddots}\pushoutcorner&\\
&&&&&&&&
}
}
\end{equation}
\caption{The Barratt--Puppe sequence of $\pi_{[1]}^\ast x$}
\label{fig:BP-constant}
\end{figure}
While $\pi^\ast$ points at the constant morphism in the middle of \autoref{fig:BP-constant}, for every $n$ the remaining $2n$-th adjoints to $\pi^\ast$ classify suitable iterated rotations of this morphism.
\end{rmk}

\bibliographystyle{alpha}
\bibliography{stability}

\end{document}

%% file: decls.tex
\usepackage{amssymb,amsmath,stmaryrd,mathrsfs}
\def\definetac{\newif\iftac}    
\ifx\tactrue\undefined
  \definetac
  \ifx\state\undefined\tacfalse\else\tactrue\fi
\fi
\def\definecref{\newif\ifcref}
\ifx\creftrue\undefined
  \definecref
  \creffalse
\fi
\iftac\else\usepackage{amsthm}\fi
\usepackage[all,2cell]{xy}
\UseAllTwocells
\usepackage{enumitem}
\usepackage{xcolor}
\definecolor{darkgreen}{rgb}{0,0.45,0} 
\usepackage[pagebackref,colorlinks,citecolor=darkgreen,linkcolor=darkgreen]{hyperref}
\usepackage{mathtools}          
\usepackage{tikz}
\usepackage{braket}             

\usepackage{url}                
\usepackage{xspace}             
\ifcref\usepackage{cleveref,aliascnt}\fi

\makeatletter
\let\ea\expandafter

\def\mdef#1#2{\ea\ea\ea\gdef\ea\ea\noexpand#1\ea{\ea\ensuremath\ea{#2}\xspace}}
\def\alwaysmath#1{\ea\ea\ea\global\ea\ea\ea\let\ea\ea\csname your@#1\endcsname\csname #1\endcsname
  \ea\def\csname #1\endcsname{\ensuremath{\csname your@#1\endcsname}\xspace}}

\DeclareRobustCommand\widecheck[1]{{\mathpalette\@widecheck{#1}}}
\def\@widecheck#1#2{%
    \setbox\z@\hbox{\m@th$#1#2$}%
    \setbox\tw@\hbox{\m@th$#1%
       \widehat{%
          \vrule\@width\z@\@height\ht\z@
          \vrule\@height\z@\@width\wd\z@}$}%
    \dp\tw@-\ht\z@
    \@tempdima\ht\z@ \advance\@tempdima2\ht\tw@ \divide\@tempdima\thr@@
    \setbox\tw@\hbox{%
       \raise\@tempdima\hbox{\scalebox{1}[-1]{\lower\@tempdima\box
\tw@}}}%
    {\ooalign{\box\tw@ \cr \box\z@}}}

\newcount\foreachcount

\def\foreachletter#1#2#3{\foreachcount=#1
  \ea\loop\ea\ea\ea#3\@alph\foreachcount
  \advance\foreachcount by 1
  \ifnum\foreachcount<#2\repeat}

\def\foreachLetter#1#2#3{\foreachcount=#1
  \ea\loop\ea\ea\ea#3\@Alph\foreachcount
  \advance\foreachcount by 1
  \ifnum\foreachcount<#2\repeat}

\def\definescr#1{\ea\gdef\csname s#1\endcsname{\ensuremath{\mathscr{#1}}\xspace}}
\foreachLetter{1}{27}{\definescr}
\def\definecal#1{\ea\gdef\csname c#1\endcsname{\ensuremath{\mathcal{#1}}\xspace}}
\foreachLetter{1}{27}{\definecal}
\def\definebold#1{\ea\gdef\csname b#1\endcsname{\ensuremath{\mathbf{#1}}\xspace}}
\foreachLetter{1}{27}{\definebold}
\def\definebb#1{\ea\gdef\csname l#1\endcsname{\ensuremath{\mathbb{#1}}\xspace}}
\foreachLetter{1}{27}{\definebb}
\def\definefrak#1{\ea\gdef\csname f#1\endcsname{\ensuremath{\mathfrak{#1}}\xspace}}
\foreachletter{1}{9}{\definefrak} 
\foreachletter{10}{27}{\definefrak}
\def\definebar#1{\ea\gdef\csname #1bar\endcsname{\ensuremath{\overline{#1}}\xspace}}
\foreachLetter{1}{27}{\definebar}
\foreachletter{1}{8}{\definebar} 
\foreachletter{9}{15}{\definebar} 
\foreachletter{16}{27}{\definebar}
\def\definetil#1{\ea\gdef\csname #1til\endcsname{\ensuremath{\widetilde{#1}}\xspace}}
\foreachLetter{1}{27}{\definetil}
\foreachletter{1}{27}{\definetil}
\def\definehat#1{\ea\gdef\csname #1hat\endcsname{\ensuremath{\widehat{#1}}\xspace}}
\foreachLetter{1}{27}{\definehat}
\foreachletter{1}{27}{\definehat}
\def\definechk#1{\ea\gdef\csname #1chk\endcsname{\ensuremath{\widecheck{#1}}\xspace}}
\foreachLetter{1}{27}{\definechk}
\foreachletter{1}{27}{\definechk}
\def\defineul#1{\ea\gdef\csname u#1\endcsname{\ensuremath{\underline{#1}}\xspace}}
\foreachLetter{1}{27}{\defineul}
\foreachletter{1}{27}{\defineul}

\def\autofmt@n#1\autofmt@end{\mathrm{#1}}
\def\autofmt@b#1\autofmt@end{\mathbf{#1}}
\def\autofmt@l#1#2\autofmt@end{\mathbb{#1}\mathsf{#2}}
\def\autofmt@c#1#2\autofmt@end{\mathcal{#1}\mathit{#2}}
\def\autofmt@s#1#2\autofmt@end{\mathscr{#1}\mathit{#2}}
\def\autofmt@f#1\autofmt@end{\mathsf{#1}}
\def\autofmt@u#1\autofmt@end{\underline{\smash{\mathsf{#1}}}}
\def\autofmt@U#1\autofmt@end{\underline{\underline{\smash{\mathsf{#1}}}}}
\def\autofmt@h#1\autofmt@end{\widehat{#1}}
\def\autofmt@r#1\autofmt@end{\overline{#1}}
\def\autofmt@t#1\autofmt@end{\widetilde{#1}}
\def\autofmt@k#1\autofmt@end{\check{#1}}

\def\auto@drop#1{}
\def\autodef#1{\ea\ea\ea\@autodef\ea\ea\ea#1\ea\auto@drop\string#1\autodef@end}
\def\@autodef#1#2#3\autodef@end{%
  \ea\def\ea#1\ea{\ea\ensuremath\ea{\csname autofmt@#2\endcsname#3\autofmt@end}\xspace}}
\def\autodefs@end{blarg!}
\def\autodefs#1{\@autodefs#1\autodefs@end}
\def\@autodefs#1{\ifx#1\autodefs@end%
  \def\autodefs@next{}%
  \else%
  \def\autodefs@next{\autodef#1\@autodefs}%
  \fi\autodefs@next}

\DeclareSymbolFont{bbold}{U}{bbold}{m}{n}
\DeclareSymbolFontAlphabet{\mathbbb}{bbold}

\newcommand{\bbone}{\ensuremath{\mathbbb{1}}\xspace}

\mdef\delbar{\overline{\partial}}

\mdef\hf{\textstyle\frac12 }
\mdef\thrd{\textstyle\frac13 }
\mdef\qtr{\textstyle\frac14 }

\newcommand{\op}{^{\mathrm{op}}}

\let\adj\dashv
\SelectTips{cm}{}
\newdir{ >}{{}*!/-10pt/@{>}}    
\newcommand{\pushoutcorner}[1][dr]{\save*!/#1+1.2pc/#1:(1,-1)@^{|-}\restore}
\newcommand{\pullbackcorner}[1][dr]{\save*!/#1-1.2pc/#1:(-1,1)@^{|-}\restore}

\mdef\Id{\mathrm{Id}}
\mdef\id{\mathrm{id}}
\alwaysmath{ell}
\alwaysmath{infty}
\alwaysmath{odot}
\def\frc#1/#2.{\frac{#1}{#2}}   
\mdef\ten{\mathrel{\otimes}}

\mdef\sqten{\mathrel{\boxtimes}}

\DeclareRobustCommand\widecheck[1]{{\mathpalette\@widecheck{#1}}}
\def\@widecheck#1#2{%
    \setbox\z@\hbox{\m@th$#1#2$}%
    \setbox\tw@\hbox{\m@th$#1%
       \widehat{%
          \vrule\@width\z@\@height\ht\z@
          \vrule\@height\z@\@width\wd\z@}$}%
    \dp\tw@-\ht\z@
    \@tempdima\ht\z@ \advance\@tempdima2\ht\tw@ \divide\@tempdima\thr@@
    \setbox\tw@\hbox{%
       \raise\@tempdima\hbox{\scalebox{1}[-1]{\lower\@tempdima\box
\tw@}}}%
    {\ooalign{\box\tw@ \cr \box\z@}}}

\DeclareMathOperator\colim{colim}

\newcommand{\too}[1][]{\ensuremath{\overset{#1}{\longrightarrow}}}

\mdef\we{\overset{\sim}{\longrightarrow}}
\mdef\leftwe{\overset{\sim}{\longleftarrow}}

\let\xto\xrightarrow

\def\rightarrowtailfill@{\arrowfill@{\Yright\joinrel\relbar}\relbar\rightarrow}
\newcommand\xrightarrowtail[2][]{\ext@arrow 0055{\rightarrowtailfill@}{#1}{#2}}

\def\twoheadrightarrowfill@{\arrowfill@{\relbar\joinrel\relbar}\relbar\twoheadrightarrow}
\newcommand\xtwoheadrightarrow[2][]{\ext@arrow 0055{\twoheadrightarrowfill@}{#1}{#2}}

\def\slashedarrowfill@#1#2#3#4#5{%
  $\m@th\thickmuskip0mu\medmuskip\thickmuskip\thinmuskip\thickmuskip
   \relax#5#1\mkern-7mu%
   \cleaders\hbox{$#5\mkern-2mu#2\mkern-2mu$}\hfill
   \mathclap{#3}\mathclap{#2}%
   \cleaders\hbox{$#5\mkern-2mu#2\mkern-2mu$}\hfill
   \mkern-7mu#4$%
}
\def\rightslashedarrowfill@{%
  \slashedarrowfill@\relbar\relbar\mapstochar\rightarrow}
\newcommand\xslashedrightarrow[2][]{%
  \ext@arrow 0055{\rightslashedarrowfill@}{#1}{#2}}
\mdef\hto{\xslashedrightarrow{}}
\mdef\htoo{\xslashedrightarrow{\quad}}

\def\toiso{\xto{\smash{\raisebox{-.5mm}{$\scriptstyle\sim$}}}}

\long\def\my@drawfill#1#2;{%
\@skipfalse
\fill[#1,draw=none] #2;
\@skiptrue
\draw[#1,fill=none] #2;
}
\newif\if@skip
\newcommand{\skipit}[1]{\if@skip\else#1\fi}
\newcommand{\drawfill}[1][]{\my@drawfill{#1}}

\newif\ifhyperref
\@ifpackageloaded{hyperref}{\hyperreftrue}{\hyperreffalse}
\iftac
  \let\your@state\state
  \def\state#1{\my@state#1}
  \def\my@state#1.{\gdef\currthmtype{#1}\your@state{#1.}}
  \let\your@staterm\staterm
  \def\staterm#1{\my@staterm#1}
  \def\my@staterm#1.{\gdef\currthmtype{#1}\your@staterm{#1.}}
  \let\@defthm\newtheorem
  \def\switchtotheoremrm{\let\@defthm\newtheoremrm}
  \def\defthm#1#2#3{\@defthm{#1}{#2}} 
  \let\your@section\section
  \def\section{\gdef\currthmtype{section}\your@section}
  \def\currthmtype{}
  \ifhyperref
    \def\autoref#1{\ref*{label@name@#1}~\ref{#1}}
  \else
    \def\autoref#1{\ref{label@name@#1}~\ref{#1}}
  \fi
  \AtBeginDocument{%
    \let\old@label\label%
    \def\label#1{%
      {\let\your@currentlabel\@currentlabel%
        \edef\@currentlabel{\currthmtype}%
        \old@label{label@name@#1}}%
      \old@label{#1}}
  }
  \let\cref\autoref
\else\ifcref
  \def\defthm#1#2#3{%
    \newaliascnt{#1}{thm}
    \newtheorem{#1}[#1]{#2}
    \aliascntresetthe{#1}
    \crefname{#1}{#2}{#3}
  }
  \let\autoref\cref  
\else
  \ifhyperref
    \def\defthm#1#2#3{
      \newtheorem{#1}{#2}[section]%
      \expandafter\def\csname #1autorefname\endcsname{#2}%
      \expandafter\let\csname c@#1\endcsname\c@thm}
  \else
    \def\defthm#1#2#3{\newtheorem{#1}[thm]{#2}} 
    \ifx\SK@label\undefined\let\SK@label\label\fi
    \let\old@label\label
    \let\your@thm\@thm
    \def\@thm#1#2#3{\gdef\currthmtype{#3}\your@thm{#1}{#2}{#3}}
    \def\currthmtype{}
    \def\label#1{{\let\your@currentlabel\@currentlabel\def\@currentlabel%
        {\currthmtype~\your@currentlabel}%
        \SK@label{#1@}}\old@label{#1}}
    \def\autoref#1{\ref{#1@}}
  \fi
  \let\cref\autoref
\fi\fi

\newtheorem{thm}{Theorem}[section]
\ifcref
  \crefname{thm}{Theorem}{Theorems}
\else
  
\fi
\defthm{cor}{Corollary}{Corollaries}
\defthm{prop}{Proposition}{Propositions}
\defthm{lem}{Lemma}{Lemmas}
\defthm{sch}{Scholium}{Scholia}
\defthm{assume}{Assumption}{Assumptions}
\defthm{claim}{Claim}{Claims}
\defthm{conj}{Conjecture}{Conjectures}
\defthm{hyp}{Hypothesis}{Hypotheses}
\iftac\switchtotheoremrm\else\theoremstyle{definition}\fi
\defthm{defn}{Definition}{Definitions}
\defthm{notn}{Notation}{Notations}
\defthm{con}{Construction}{Constructions}
\iftac\switchtotheoremrm\else\theoremstyle{remark}\fi
\defthm{rmk}{Remark}{Remarks}
\defthm{eg}{Example}{Examples}
\defthm{egs}{Examples}{Examples}
\defthm{ex}{Exercise}{Exercises}
\defthm{ceg}{Counterexample}{Counterexamples}

\ifcref
  \crefformat{section}{\S#2#1#3}
  \Crefformat{section}{Section~#2#1#3}
  \crefrangeformat{section}{\S\S#3#1#4--#5#2#6}
  \Crefrangeformat{section}{Sections~#3#1#4--#5#2#6}
  \crefmultiformat{section}{\S\S#2#1#3}{ and~#2#1#3}{, #2#1#3}{ and~#2#1#3}
  \Crefmultiformat{section}{Sections~#2#1#3}{ and~#2#1#3}{, #2#1#3}{ and~#2#1#3}
  \crefrangemultiformat{section}{\S\S#3#1#4--#5#2#6}{ and~#3#1#4--#5#2#6}{, #3#1#4--#5#2#6}{ and~#3#1#4--#5#2#6}
  \Crefrangemultiformat{section}{Sections~#3#1#4--#5#2#6}{ and~#3#1#4--#5#2#6}{, #3#1#4--#5#2#6}{ and~#3#1#4--#5#2#6}
  \crefformat{appendix}{Appendix~#2#1#3}
  \Crefformat{appendix}{Appendix~#2#1#3}
  \crefrangeformat{appendix}{Appendices~#3#1#4--#5#2#6}
  \Crefrangeformat{appendix}{Appendices~#3#1#4--#5#2#6}
  \crefmultiformat{appendix}{Appendices~#2#1#3}{ and~#2#1#3}{, #2#1#3}{ and~#2#1#3}
  \Crefmultiformat{appendix}{Appendices~#2#1#3}{ and~#2#1#3}{, #2#1#3}{ and~#2#1#3}
  \crefrangemultiformat{appendix}{Appendices~#3#1#4--#5#2#6}{ and~#3#1#4--#5#2#6}{, #3#1#4--#5#2#6}{ and~#3#1#4--#5#2#6}
  \Crefrangemultiformat{appendix}{Appendices~#3#1#4--#5#2#6}{ and~#3#1#4--#5#2#6}{, #3#1#4--#5#2#6}{ and~#3#1#4--#5#2#6}
  \crefformat{subappendix}{\S#2#1#3}
  \Crefformat{subappendix}{Section~#2#1#3}
  \crefrangeformat{subappendix}{\S\S#3#1#4--#5#2#6}
  \Crefrangeformat{subappendix}{Sections~#3#1#4--#5#2#6}
  \crefmultiformat{subappendix}{\S\S#2#1#3}{ and~#2#1#3}{, #2#1#3}{ and~#2#1#3}
  \Crefmultiformat{subappendix}{Sections~#2#1#3}{ and~#2#1#3}{, #2#1#3}{ and~#2#1#3}
  \crefrangemultiformat{subappendix}{\S\S#3#1#4--#5#2#6}{ and~#3#1#4--#5#2#6}{, #3#1#4--#5#2#6}{ and~#3#1#4--#5#2#6}
  \Crefrangemultiformat{subappendix}{Sections~#3#1#4--#5#2#6}{ and~#3#1#4--#5#2#6}{, #3#1#4--#5#2#6}{ and~#3#1#4--#5#2#6}
  \crefname{part}{Part}{Parts}
  \crefname{figure}{Figure}{Figures}
\fi

\def\thmqedhere{\expandafter\csname\csname @currenvir\endcsname @qed\endcsname}

\renewcommand{\theenumi}{(\roman{enumi})}
\renewcommand{\labelenumi}{\theenumi}

\setitemize[1]{leftmargin=2em}
\setenumerate[1]{leftmargin=*}

\iftac
  \let\c@equation\c@subsection
\else
  \let\c@equation\c@thm
\fi
\numberwithin{equation}{section}

\alwaysmath{alpha}
\alwaysmath{beta}
\alwaysmath{gamma}
\alwaysmath{Gamma}
\alwaysmath{delta}
\alwaysmath{Delta}
\alwaysmath{epsilon}
\mdef\ep{\varepsilon}
\alwaysmath{zeta}
\alwaysmath{eta}
\alwaysmath{theta}
\alwaysmath{Theta}
\alwaysmath{iota}
\alwaysmath{kappa}
\alwaysmath{lambda}
\alwaysmath{Lambda}
\alwaysmath{mu}
\alwaysmath{nu}
\alwaysmath{xi}
\alwaysmath{pi}
\alwaysmath{rho}
\alwaysmath{sigma}
\alwaysmath{Sigma}
\alwaysmath{tau}
\alwaysmath{upsilon}
\alwaysmath{Upsilon}
\alwaysmath{phi}
\alwaysmath{Pi}
\alwaysmath{Phi}
\mdef\ph{\varphi}
\alwaysmath{chi}
\alwaysmath{psi}
\alwaysmath{Psi}
\alwaysmath{omega}
\alwaysmath{Omega}

\makeatother
